\documentclass[10pt]{amsart}
\usepackage{hyperref}
\usepackage{amsmath,amsfonts,mathrsfs,amssymb,latexsym,array,mathtools}
\usepackage[all]{xy}
\newcommand{\N}{\mathcal{S}} 
\newcommand{\CM}{{\overline{\mathcal{M}_{g,1}^{\N}}}} 
\newcommand{\M}{{\mathcal{M}_{g,1}^{\N}}} 
\newcommand{\C}{\mathcal{C}} 
\renewcommand{\k}{\mathbf{k}} 
\newcommand{\proj}{\mathbb{P}} 
\renewcommand{\l}{\ell} 

\DeclareMathOperator*{\codim}{codim} 
\DeclareMathOperator*{\GL}{GL}

\newtheorem{thm}{Theorem}[section]
\newtheorem{lem}[thm]{Lemma}

\newtheorem{prop}[thm]{Proposition}
\newtheorem{rmk}[thm]{Remark}
\newtheorem{syzlem}[thm]{Syzygy Lemma}
\def\c[#1,#2]{{c_{{#1},#2}}}


\begin{document}

\title{On the locus of curves with an odd subcanonical marked point}

\author{Andr\'{e} Contiero}
\address{Departamento de Matem\'atica, ICEx, UFMG. Av. Ant\^onio Carlos 6627, 30123-970 Belo Horizonte MG, Brazil}
\email{contiero@ufmg.br\footnote{Partially supported by \emph{Universal CNPq 486468/2013-5 } and by \emph{Programa Institucional de Aux\'ilio \`a Pesquisa de Docentes Rec\'em-Contratados - UFMG}}
}

\author{Aislan Leal Fontes}
\address{Departamento de Matem\'atica, UFS - Campus Itabaiana. Av. Vereador Ol\'impio Grande s/n, 49506-036 Itabaiana/SE, Brazil}
\email{aislan@ufs.br}


\begin{abstract}

We present an explicit construction of a compactification of the locus 
of smooth curves whose symmetric Weierstrass semigroup 
at a marked point is odd. The construction is an extension of Stoehr's techniques using Pinkham's equivariant deformation of 
monomial curves by exploring syzygies. As an application we prove the rationality of the locus for genus at most six.

\end{abstract}

\maketitle

\section{Introduction}


Let $\mathscr{H}_{2g-2}$ be the locus of compact Riemann surfaces (smooth projective algebraic curves) of genus $g>1$ with a 
fixed abelian differential vanishing at a point to order $2g-2$. In a remarkable work Kontsevich--Zorich \cite[Thm. 1]{KZ03} 
showed that $\mathscr{H}_{2g-2}$ has exactly three irreducible components, namely the locus
$\mathscr{H}_{2g-2}^{\mathrm{hyp}}$ of hyperelliptic points, the even $\mathscr{H}_{2g-2}^{\mathrm{even}}$ and the odd 
$\mathscr{H}_{2g-2}^{\mathrm{odd}}$ points. Ten years later Bullock \cite[Thm. 2.1]{Bu13} characterized a general point of each component,
a general point of $\mathscr{H}_{2g-2}^{\mathrm{hyp}}$ has Weierstrass gaps $\{1,3,5,\ldots,2g-3,2g-1\}$, 
a general point of $\mathscr{H}_{2g-2}^{\mathrm{odd}}$ has Weierstrass semigroup $\{1,2,3,\ldots,g-1,2g-1\}$ and finally
a general point of $\mathscr{H}_{2g-2}^{\mathrm{even}}$ has Weierstrass gaps $\{1,2,3,\ldots,g-2,g,2g-1\}$. 
Say that an abelian di\-ffe\-rential with a zero at a point of order $2g-2$ it is equivalent to required that this point is \textit{subcanonical}, 
\cite[Def. 1]{Bu13}, i.e. the associated Weierstrass semigroup of this point is \textit{symmetric}.

Let $\M$ be the moduli space of smooth pointed curves of genus $g>1$ with a fixed symmetric Weierstrass
semigroup $\N$ at the marked point. There are two powerful tools to investigate the moduli spaces $\M$, both based on deformation of suitable curves.
On the one hand Eisenbud--Harris \cite{EH87}
deformed stable curves and uses limit linear series
to study properties of $\M$ as a locally closed subset of $\mathcal{M}_{g,1}$. 
On the other hand, Pinkham \cite{Pi74} studied the moduli $\M$ by using equivariant deformation theory, 
deforming \textit{monomial curves}. Following a proposal given by Mumford \cite{M75} on 
Petri's analysis of the canonical ideal, Stoehr \cite{St93} constructed a compactification of $\M$ when $\N$ is symmetric by allowing 
Gorenstein curves at its bordering. 
Stoehr's techniques avoid suitable classes of symmetric semigroups, more precisely, it is assumed that the multiplicity  $n_1$ of $\N$ satisfies 
$3<n_1<g$, and that $\N\neq \langle 4,5\rangle$, avoiding the general points of $\mathscr{H}_{2g-2}^{\mathrm{hyp}}$ and of 
$\mathscr{H}_{2g-2}^{\mathrm{odd}}$  of the Kontsevich--Zorich space $\mathscr{H}_{2g-2}$. Another successful approach to study families of 
Weierstrass points can be done by considering (generalized) Wronskians and its derivatives, we refer to \cite{Ga,GP99}.

In this paper we extend Stoehr's techniques in order to construct in a rather explicit way a compactification $\overline{\M}$ of the moduli
space $\M$ when $\N$ is a symmetric semigroup different from the hyperelliptic $\langle 2,2g+1\rangle$. Numerical semigroups of odd type 
tends to be realized as Weierstrass semigroups of possibly singular Gorenstein curves which are a triple recovering of the 
projective line $\mathbb{P}^1$, ie $3$-gonal singular curves, see Lemma \ref{lem31} below. Hence the canonical 
ideal of the monomial Gorenstein curve associated to a numerical odd semigroup cannot be generated by only quadratic forms as required, 
cf. Lemma \ref{lemaI0} of the present work.

Given a nonhyperelliptic symmetric semigroup $\N\neq\langle 2,2g+1\rangle$, we deform the ideal (which is generate by quadratic and 
cubic forms) of the associated canonically embedded monomial Gorenstein curve. By analyzing suitable syzygies of canonical ideals, see Lemma \ref{lem2}, 
we get a compactification of $\M$ by allowing Gorenstein singularities at its bordering, cf. Theorem \ref{teo3}. 
The compactification is (by construction)
a closed subset of the weighted projective space $\mathbb{P}(\mathrm{T}^{1,-}(\k[\N]))$, where $\mathrm{T}^{1,-}(\k[\N])$ stands for the
negatively graded part of the first module of cotangent complex associated to a suitable monomial curve. Since our construction is completely
explicit, we are able to produce non-trivial examples and investigate the global geometry of the moduli spaces $\M$. In the last 
section of this paper we illustrate our techniques computing the equations of $\overline{\M}$ when $\N$ is odd of genus $5$,
$\N=\langle 5,6,7,8\rangle$ and of genus $6$, $\N=\langle 6,7,8,9,10\rangle$.

\section{Gorenstein subcanonical curves and Weierstrass Points}

Let $\C$ be a complete integral Gorenstein curve of arithmetical genus $g>1$ defined over an algebraically field $\k$.
Throughout this section we assume that $\C$ is subcanonical, i.e.  there is a rational function on $\C$ 
with pole divisor $(2g-2)P$,  where $P$ is a nonsingular point of $\C$. The dualizing sheaf $\omega$ of $\C$ is $\mathcal{O}_{\C}((2g-2)P)$, 
and the vector space of its global sections is
$$H^0(\C, \omega)=\k\cdot x_{n_0}\oplus \k\cdot x_{n_1}\oplus\dots\oplus\k\cdot x_{n_{g-1}}$$
where $x_{n_i}$ is a rational function on $\C$ whose pole divisor is $n_iP$, for $i\geq 1$, with $n_0:=0$ and $n_{g-1}=2g-2$.
Equivalently, the base point $P\in \C$ is a Weierstrass point with gap sequence $1=\l_1<\l_2<\dots<\l_g=2g-1$, whose
symmetric Weierstrass semigroup $\N$ of genus $g$ is canonically
generated by its first $g$ non-gaps, $\langle n_0,n_1,\dots, n_{g-1}\rangle=\N$. 
We recall that a semigroup $\N$ of genus $g$ is symmetric if its Frobenius number $\l_g$ 
is the largest possible, namely $\l_{g}=2g-1$. Equivalently, $\N$ is symmetric if and only if $\l_i=\l_g-n_{g-i}$, for all $i=1,\dots, g$.

Let us assume that $\C$ is also non-hyperelliptic, thus its dualizing sheaf $\omega$ 
induces an embedding in the $(g-1)$-dimensional projective space $\mathbb{P}^{g-1}$ defined over $\k$,
$$(x_{n_o}:\dots:x_{n_{g-1}}):\C\xhookrightarrow { \ \ \omega \ \ }\mathbb{P}^{g-1}=\mathbb{P}(H^0(\C, \omega))\,.$$
 Therefore, $\C$ can be identified with its image under the canonical embedding. Hence
$\C\subset\mathbb{P}^{g-1}$ is a projective curve of genus $g$ and degree $2g-2$.

Reciprocally, every nonhyperelliptic symmetric numerical semigroup $\N$ of genus $g>1$ can be rea\-lized as a Weierstrass semigroup of a canonical Gorenstein curve. 
We just consider the canonical generators $0=n_0<n_1,\dots,<n_{g-1}=2g-2$ of $\N$ and take the induced (canonical) monomial curve
$$\C^{(0)}:=\{(s^{n_0}t^{\l_g-1}: s^{n_1}t^{\l_{g-1}-1}:\ldots:s^{n_{g-2}}t^{\l_2-1}: s^{n_{g-1}}t^{\l_1-1})\,\vert\,(s:t)\in\mathbb{P}^1\}\subset \mathbb{P}^{g-1}\,.$$
It can be checked that it has a unique singular point, namely $(1:0:\dots:0)$, which is unibranch and has singularity degree $g$.
Since the semigroup $\N$ is symmetric, $C^{(0)}$ is a Gorenstein curve. The contact orders with
hyperplanes at its unique point $P=(0:\dots0:1)$ at the infinity are exactly $\l_i-1$, $i=1,\dots,g$ (the vanishing sequence). Thus $\C^{(0)}$ has degree $2g-2$ and its 
Weierstrass semigroup at $P$ is $\N$. 

According to Enriques--Babbage's theorem for smooth curves, cf. \cite{ACGH85},  
if we assume that $\C$ is not isomorphic to a plane quintic, then its ideal can be generated by quadratic forms, 
when it is non-trigonal, and by quadratic and cubic forms when it is trigonal. 

An extended version of Max Noether's theorem for complete integral 
non-hy\-pereliptic curves, 
cf. \cite{CFM, Ma},
states that there is a surjective homomorphism $$\mathrm{Sym}^r(H^0(\C, \omega))\longrightarrow H^0(\C, \omega^{r})$$ for all $r\geq 1$.
In the following, we recall a suitable proof of Max-Noether's theorem for subcanonical curves given by St\"ohr in \cite{St93}.

Let $\C$ be a complete non-hyperelliptic Gorenstein curve with a subcanonical point $P$. Since $\C$ is non-hyperelliptic, we must to assume that the symmetric
semigroup $\N$ is not hyperelliptic, i.e.  $2\notin\N$, equivalently $\N\neq\langle2,2g+1\rangle$.
Now, for each nongap $s\leq 4g-4$, we consider the partitions of $s$ as sums of two nongaps,
$$s=a_{s}+b_{s},\ a_{s}\leq b_{s}\leq 2g-2,$$
with $a_s$ the smallest possible nongap. From Oliveira's paper \cite[Thm. 1.3]{O91} the following $3g-3$ rational 
functions $x_{a_s}x_{b_s}$ of $\C$ form a $P$-hermitian basis for the space
of the global sections of the bicanonical divisor $\omega^{2}=\mathcal{O}_{\C}((4g-4)P)$. 
Now, for each integer $r\geq 3$ a $P$-hermitian basis for the space $H^{0}(\C, \omega^{r})$ is given by the $r$-monomials expressions 
\begin{equation*}
 \begin{array}{lr}
 x_{n_{0}}^{r-1}x_{n_i} & (i=0,\ldots,g-1),\\
 x_{n_{0}}^{r-2-i}x_{a_s}x_{b_s}x_{n_{g-1}}^{i} & (i=0,\ldots, r-2,\ s=2g,\ldots,4g-4),\\
 x_{n_{0}}^{r-3-i}x_{n_1}x_{2g-n_1}x_{n_{g-2}}x_{n_{g-1}}^{i} & (i=0,\ldots, r-3).
 \end{array} 
 \end{equation*}
 
\noindent Note that the pole orders of the above $(2r-1)(g-1)$ rational functions are pairwise different, so they form a linearly independent set in
$H^0(\C, \omega^r)$.

  
 Let $I(\C)=\displaystyle\oplus_{r=2}^{\infty}I_r(\C)$ be the homogeneous canonical ideal of $\C\subset\mathbb{P}^{g-1}$. As an 
 immediate consequence of the existence of the above 
 $P$-hermitian basis of $r$-monomials for $H^{0}(\C, \omega^r)$, the homomorphism 
  \begin{equation*}
  {\mathbf{k}[X_{n_0},\ldots,X_{n_{g-1}}]}_r\longrightarrow H^{0}(\C, \omega^r)
  \end{equation*}
 induced by the substitutions $X_{n_{i}}\longmapsto x_{n_i}$ is surjective for each $r\geq 1$. Thus we get a proof of Max-Noether's 
 theorem for non-hyperelliptic Gorenstein curves with a subcanonical point.
 
 By virtue of Riemann's theorem, for each $r\geq 2$, the codimension of $I_r({\C})$ in the ${r+g-1\choose r}$-dimensional vector 
 space ${\mathbf{k}[X_{n_0},\ldots,X_{n_{g-1}}]}_r$ of homogeneous $r$-forms is equal to $(2r-1)(g-1)$.
So the vector space of quadratic and cubic relations have dimensions 
  \begin{equation*}
  \dim I_2(\C)=\frac{(g-2)(g-3)}{2} \ \ \mbox{and}\ \dim I_3(\C)={g+2\choose 3}-(5g-5),
  \end{equation*}
  respectively. 
  
 For each $r\geq 2$, we define the vector subspace $\Lambda_r$ of ${\mathbf{k}[X_{n_0},\ldots,X_{n_{g-1}}]}_r$ 
spanned by the lifting of the above $P$-hermitian $r$-monomial basis of $H^{0}(\C, \omega^r)$. It is spanned by the $r$-monomials in $X_{n_0},\ldots,X_{n_{g-1}}$ 
  whose weights are pairwise different between all the nongaps $n\leq r(2g-2)$. Since $\Lambda_r\cap I_r(\C)=0$ and
  \begin{equation*}
  \dim\Lambda_r=\dim H^{0}(\C, \omega^r)=\codim I_r(\C),
  \end{equation*}
  we obtain
  \begin{equation}\label{directsum}
  {\mathbf{k}[X_{n_0},\ldots,X_{n_{g-1}}]}_r=I_r(\C)\oplus\Lambda_r, \ \mbox{for each}\ r\geq 2.
  \end{equation}
 Let $r\N$ be the set of all sums of $r$ nongaps not bigger than $2g-2$. Oliveira showed, cf. \cite[theorem 1.5]{O91}, that each nongap smaller than or 
  equal to $r(2g-2)$ belongs to $r\N$. Moreover, each sum of $r$ nongaps $\leq 2g-2$ is a nongap $\leq r(2g-2)$. 
  Consequently, $\#r\N=(2r-1)(g-1)$
and therefore $$\#r\N=\dim H^{0}(\C, \omega^r)\,.$$ In particular, for each nongap $s\leq 4g-4$ we list all the partitions $s=a_{si}+b_{si}\in 2\N$, where 
  \begin{align*}
  a_{si}\leq b_{si}\leq 2g-2\ (i=0,\ldots,\nu_s)\ \mbox{and}\ a_s:=a_{s0}<a_{s1}<a_{s2}<\ldots<a_{s\nu_s}\,.
  \end{align*}
  Since $x_{a_{si}}x_{b_{si}}\in H^{0}(\C, sP)$ and $\{x_{a_{s}}x_{b_{s}}\}$ is the above fixed basis, we can write 
  \begin{equation*}
 x_{a_{si}}x_{b_{si}}=\displaystyle\sum_{n=0}^{s} c_{sin}x_{a_{n}}x_{b_{n}},
  \end{equation*} for each $i=0,\dots,\nu_s$, where the coefficients $c_{sir}$ are uniquely determined constants and the summation 
  index only varies through nongaps. In the same way, 
  for each nongap $\sigma\leq 6g-6$ we consider the partitions $\sigma=a_{\sigma j}+b_{\sigma j}+c_{\sigma j}\in 3\N$ where 
  $a_{\sigma j}\leq b_{\sigma j}\leq c_{\sigma j}\leq 2g-2\ (j=0,\ldots,\nu_\sigma)$ with 
$a_{\sigma}:=a_{\sigma0}<a_{\sigma 1}<\ldots<a_{\sigma\nu_\sigma}$ and $b_{\sigma}:=b_{\sigma 0}>b_{\sigma 1}>\ldots>b_{\sigma\nu_\sigma}$.
Analogously, we can write
   \begin{equation*}
  x_{a_{\sigma j}}x_{b_{\sigma j}}x_{c_{\sigma j}}=\displaystyle\sum_{n=0}^{\sigma} d_{\sigma jn}x_{a_{n}}x_{b_{n}x_{c_{n}}},
   \end{equation*} for each integer $j=0,\ldots,\nu_{\sigma}$, where the coefficients $d_{\sigma jn}$ are uniquely determined constants and the 
   summation index only varies through nongaps.
   
 Multiplying the functions $x_{n_0},\ldots, x_{n_{g-1}}$ by constants we do not change the $P$-hermitian property of the above basis, thus
 we can normalize the coefficients $c_{sis}=1$ and $d_{\sigma j\sigma}=1$.
 Therefore, the $\binom{g+1}{2}-(3g-3)=\frac{1}{2}(g-3)(g-2)$ quadratic forms
  \begin{equation}\label{qdforms}
  F_{si}=X_{a_{si}}X_{b_{si}}-X_{a_s}X_{b_s}-\displaystyle\sum_{n=0}^{s-1}c_{sin}X_{a_{n}}X_{b_{n}}
   \end{equation} 
   and the $\binom{g+2}{3}-(5g-5)$ cubic forms
   \begin{equation}\label{cforms}
    G_{\sigma j}=X_{a_{\sigma j}}X_{b_{\sigma j}}X_{c_{\sigma j}}-X_{a_{\sigma }}X_{b_{\sigma }}X_{c_{\sigma }}-\displaystyle\sum_{n=0}^{\sigma-1}d_{\sigma jn}X_{a_{n}}X_{b_{n}}X_{c_{n}},
     \end{equation}
vanish identically on $\C$. We attach to the variable $X_n$ the weight $n$, to the coefficient $c_{sin}$ the weight $s-n$ and to
$d_{\sigma jn}$ the weight $\sigma-n$. Thus the above quadric and cubic forms are also \textit{isobaric} forms.

In the view of Henriques--Babbage's theorem we want to assure that the canonical ideal of $\C$ is
generated by the above quadratic and cubic forms. We assume that the non-hyperelliptic symmetric semigroup $\N$ is a non-trivial semigroup of genus $g>1$, which is
equivalent to assume that the multiplicity $n_1$ of $\N$ satisfies $2<n_1\leq g$.
By a theorem of Oliveira \cite[Thm. 1.7]{O91}, if we consider $3<n_1<g$, then there 
is at least one quadratic form, i.e.  $\nu_s\geq 1$, whenever $s=n_i+2g-2$ for $i=0,\ldots, g-3$. 
In this case Contiero--Stoehr \cite{CS} gave an algorithmic proof that the canonical ideal of a
Gorenstein curve $\C\subset\mathbb{P}^{g-1}$ with Weierstrass semigroup $\N$ at the base point is generated by only quadratic relations. If we assume
that $3\in\N$ then its genus has residue $1$ or $0$ module $3$, hence $\N:=\langle 3,g+1\rangle$. In this case
we already know that $\CM = \proj(T^{1,-}_{\k[\N]|\k})$. 
If $\N=\langle 4,5\rangle$ , then $\C$ is isomorphic to a plane quintic where the quadric hypersurfaces containing $\C$ is the Veronese surface.

In the excluded case $S=\mathbb{N}\backslash\{1,2,\ldots,g-1, 2g-1\}$, the curve $\C$ is possibly trigonal, so its canonical ideal can not be generated by
only quadratic relations. In the next section we investigate the Weierstrass semigroup of trigonal complete curves and then, we will give an algorithmic proof that the
canonical ideal of a complete Gorenstein curve with symmetric Weierstrass semigroup 
$$S:=\mathbb{N}\backslash\{1,2,\ldots,g-1, 2g-1\}=\langle 0,g,g+1,\dots,2g-2\rangle$$ 
at a smooth non-ramified point is generated by quadric and cubics forms.

\section{Curves with odd subcanonical points} 


Let $\C$ be a complete integral curve of arithmetic genus $g$ defined over an algebraically closed field $\k$.
A \emph{linear system of dimension $r$ on $\C$} is a set of the form
$$
\mathscr{L} =\mathscr{L}(\mathscr{F} ,V):=\{x^{-1}\mathscr{F}\ |\ x\in V\setminus 0\}
$$
where $\mathscr{F}$ is a coherent fractional ideal sheaf on $C$ and $V$ is a vector subspace of 
$H^{0}(\C, \mathscr{F})$ of dimension $r+1$. 

The notion of linear systems on curves presented here
is characterized by interchanging bundles by torsion free sheaves of rank $1$. 
This is a meaningful approach since they may possess \emph{non-removable} base points, see  Coppens \cite{Cp}.

 The \emph{degree} of the linear system $\mathscr{L}$ is the integer
$\deg \mathscr{F} :=\chi (\mathscr{F} )-\chi (\mathcal{O}_{\C})$, where $\chi$ denotes the Euler characteristic. Note, in particular, that if $\mathcal{O_{\C}}\subset\mathscr{F}$ then
$$\deg\mathscr{F}=\sum_{P\in C}\dim(\mathscr{F}_P/\mathcal{O}_{\C, P}).$$
The notation $g_{d}^{r}$ stands for a linear system of degree $d$ and dimension $r$. 
The linear system is said to be \emph{complete} if $V=H^0(\C, \mathscr{F})$, in this case one simply writes $\mathscr{L}=|\mathscr{F}|$.
According to E. Ballico's \cite[p. 363, Dfn. 2.1 (3)]{Bal}, the gonality of $\C$ is the smallest $d$ for which there exists a 
$g_{d}^{1}$ on $\C$, or equivalently, a torsion free sheaf $\mathscr{F}$ of rank $1$ on $\C$ with degree $d$ and 
$h^0(\C, \mathscr{F})\geq 2$. %



The following lemma is a straightforward generalization of a Kim's result \cite[Thm. 2.6]{Kim} characterizing the Weierstrass semigroup
associated to a non-ramification point of a trigonal curve.

\begin{lem}\label{lem31}
Let $\C$ be a complete integral trigonal curve of arithmetical genus $g\geq 5$ and $P\in\C$ a Weierstrass non-ramification point. 
The Weierstrass semigroup $\N$ of $\C$ at $P$ is of the form
$$\{0,m,m+1,m+2,\dots,m+(s-g), s+2,s+3,s+4,\dots\},$$
for some $s$ and $m$ such that $g\geq m\geq\left\lfloor\frac{s+1}{2}\right\rfloor+1$.
In particular, in the symmetric case we get the odd semigroup
$$\N=\{0,g,g+1,\dots,2g-2,2g,2g+1,2g+2,\dots\}.$$

\end{lem}
\begin{proof}
Let $\l_g$ be the Frobenius number of the Weierstrass semigroup $\N$ associated to $P\in \C$. Equivalently, the integer $s:=\l_g-1$ 
is the largest such that the
divisor $D_0=s\,P$ is special. Since $P$ is a Weierstrass point, it is immediate that $g\leq \l_g-1\leq 2g-2$. By the maximality of $s$
\begin{equation*}
\dim|\mathcal{O}(D_0)|=s-g+1.
\end{equation*}
Since $D_0$ is a special divisor, let be
\begin{equation*}
\omega_{\C}\simeq \mathcal{O}_{\C}(D_0+P_1+P_2+\ldots+P_{2g-2-s})
\end{equation*}
the dualizing sheaf of $\C$ where $P_i\in\C$, $P_i\neq P$, with $i=1,\ldots,2g-2-s$. As $P$ is not a ramification point, 
the first nongap $m$ is greater than $3$, and so $|mP|$ is not compounded of $g^1_3$.
By considering the divisor $$D:=(s-m)P+P_1+P_2+\ldots+P_{2g-2-s}$$ we see that $|D|$ is compounded of 
$g^1_3$ because $\omega_{\C}=\mathcal{O}_{\C}(mP)\otimes\mathcal{O}_{\C}(D)$.

 Applying the Riemann-Roch theorem, $\dim |D|=g-m$, hence we can write $|D|=(g-m)g^1_3+B$, where $B$ is the base 
 locus of $|D|$. For each element $R$ of $g^1_3$ with $R\succeq P$, we have $R=P+Q_1+Q_2$, with $P\neq Q_1$ and $P\neq Q_2$ 
 because $P$ is not a ramification point of $\C$, thus 
 \begin{equation*}
 D=(g-m)(P+Q_1+Q_2)+B=(s-m)P+P_1+P_2+\ldots+P_{2g-2-s},
 \end{equation*} 
 and by the maximality of $s$
 \begin{equation*}
 P_1+P_2+\ldots+P_{2g-2-s}\succeq (g-m)Q_1+(g-m)Q_2, 
 \end{equation*}
 implying $2(g-m)\leq 2g-2-s$. Therefore, $m\geq\left\lfloor\frac{s+1}{2}\right\rfloor+1$. 
 
 On the other hand, 
 \begin{equation*}
 B\succeq (s-g)P,
 \end{equation*}
 which means that $(s-g)P$ is contained in the base locus of $|D|$. Consequently each divisor $iP$ is not in the base 
 locus of $|mP+iP|, i=0,\ldots, s-g$, and thus $m, m+1,\ldots, m+s-g$ are nongaps of $\N$. 
 
 Now by definition of $s$ and by Riemann-Roch theorem, $\dim|(s+1)P|=\dim |sP|$, which implies that $s+1$ is a gap of $\N$. 
 Since the divisor $(r-1)P$ is nonspecial for each integer $r\geq s+2$, it follows that
\begin{equation*}
\dim |rP|=r-g=\dim|(r-1)P|+1,
\end{equation*}
so each $r\geq s+2$ is a nongap. In this way the set
\begin{equation*}
S=\left\{0, m, m+1,\ldots, m+(s-g), s+2,\ldots\right\}
\end{equation*}
is contained in $\N$ and the cardinality of $\mathbb{N}-S$ is $g$.
\end{proof}


Let us consider the \textit{odd numerical semigroup} $\N:=\langle 0,g,g+1,\dots,2g-2\rangle$ of genus $g\geq 5$.
We now fix $\frac{1}{2}(g-3)(g-2)$ initial quadratic forms like in \eqref{qdforms}
  \begin{equation*}
  F_{si}^{(0)}:=X_{a_{si}}X_{b_{si}}-X_{a_s}X_{b_s}
   \end{equation*} and $\binom{g+2}{3}-(5g-5)$ initial cubic forms like in \eqref{cforms}
   \begin{equation*}
    G_{\sigma j}^{(0)}:=X_{a_{\sigma j}}X_{b_{\sigma j}}X_{c_{\sigma j}}-X_{a_{\sigma }}X_{b_{\sigma }}X_{c_{\sigma }}.
   \end{equation*}
   
\noindent It is clear that a considerable amount of cubic forms are just multiplies of quadratic ones, the next result explicitly find
them.
 
 \begin{prop} Let $\N:=\langle 0,g,g+1,\dots,2g-2\rangle$. There are exactly $\wp={g+2\choose 3}-(5g-5)-\eta$, with
 $$\eta=(g-3)(g-2)+(g-2)\left\lfloor\frac{g-2}{2}\right\rfloor+\left\lfloor\frac{g-3}{2}\right\rfloor+\displaystyle\sum_{j=1}^{g-4}\left\lfloor\frac{g-2-j}{2}\right\rfloor$$
initial cubic forms which are not multiples of the quadratic ones.  
 \end{prop}
 \begin{proof}
Since the fixed basis for $\Lambda_2$ is $\{X_0^2, X_0X_{g+i},X_gX_{g+i}, X_{g+j}X_{2g-2}\}$ with $i=0,\ldots,g-2$ and $j=1,\ldots,g-2$,
the initial quadratic forms are
 \begin{equation*}
F_{sl}^{(0)}=X_{a_{sl}}X_{b_{sl}}-X_gX_{g+i}\ \mbox{ and }\ F_{sl}^{(0)}=X_{a_{sl}}X_{b_{sl}}-X_{g+j}X_{2g-2},
 \end{equation*}
where the $2$-monomials nonbasis elements of $\Lambda_2$ are the products $X_{g+i}X_{g+j}$ where $1\leq i\leq j=1,\ldots, g-3$. 
While the fixed basis for $\Delta_3$ is 
$$\{X_0^2X_{i}, X_0X_{as}X_{bs}, X_{as}X_{bs}X_{2g-2},X_{g}^2X_{2g-3}\}\,,$$ with $i=0,g,g+1,\dots,2g-2$
and $\{X_{as}X_{bs}\}$ the above fixed basis for $\Delta_2$. 
Set $F:=F_{sl}^{(0)}$ for a initial quadratic form. It is clear that  the $(g-3)(g-2)$ products $X_0F$ and $X_{2g-2}F$ are cubic forms for every $F$.  
Since the monomials $X_{g+k}X_{g+i}X_{g+j}\notin\Lambda_3$ for $k=0,\ldots, g-3$ and $i,j=1,\ldots,g-3$, 
the  product $X_{g+k}F$ defines a cubic form when $X_{g+k}X_{g}X_{g+i}$ or $X_{g+k}X_{g+j}X_{2g-2}$ 
 are in $\Lambda_3$. In the first case, $X_{g+k}X_{g}X_{g+i}\in\Lambda_3$ just for $i=g-2$,
 $k=0,\ldots, g-3$ and for $i=g-3, k=0$. Hence we get the following $(g-2)\left\lfloor\frac{g-2}{2}\right\rfloor+\left\lfloor\frac{g-3}{2}\right\rfloor$ cubic forms
 $$X_{g+k}\left(X_{a_{sl}}X_{b_{sl}}-X_gX_{2g-2}\right), \text{ with }k=0,\ldots, g-3$$
 and
 $$X_g\left(X_{a_{sl}}X_{b_{sl}}-X_gX_{2g-3}\right).$$
 In the remaining case, $X_{g+k}X_{g+j}X_{2g-2}\in\Lambda_3$ just for $k=0$, $j=1,\ldots, g-2$. So we get the 
 following $\sum_{j=1}^{g-4}\left\lfloor\frac{g-2-j}{2}\right\rfloor$ initial cubic forms
 $$X_g\left(X_{a_{sl}}X_{b_{sl}}-X_{g+j}X_{2g-2}\right), \ j=1,\ldots, g-4.$$\end{proof}

It is straightforward that the quadratic $F_{si}^{(0)}$ and cubic forms $G_{\sigma j}^{(0)}$ vanish identically on the monomial curve $\C^{(0)}$.
The next lemma show that they generate the ideal of $\C^{(0)}$.
  
 \begin{lem}\label{lemaI0}
 The canonical ideal $I({\C}^{(0)})$ is generated by the $\frac{1}{2}(g-2)(g-3)$ quadratic forms $F_{si}^{(0)}$ and by the $\wp$ cubic forms $G_{\sigma j}^{(0)}$.
 \end{lem}
 \begin{proof}
 Since the $I({\C}^{(0)})$ is generated by homogeneous and isobaric forms, all we have to do is to show that for a homogeneous and isobaric form belongs to
 $I({\C}^{(0)})$ if and only if belongs to the ideal $\mathcal{J}$ generated by the binomials $F_{si}^{(0)}$ and $G_{\sigma j}^{(0)}$. It is just obvious that
 $\mathcal{J}\subseteq I({\C}^{(0)}) $. For the opposite inclusion we order the monomials $\prod_{k=0}^{g-1}X_{n_k}^{i_k}$ according to the
lexicographic ordering of the vectors $$\left(\sum\,i_k,\sum\,n_k\,i_k,-i_0,-i_{g-1},\dots,-i_1\right).$$ In this way, the binomials $F_{si}^{(0)}$ and $G_{\sigma j}^{(0)}$
form a Groebner basis for $\mathcal{J}$. Now, for each homogeneous form $F$ of degree $r$ which is also isobaric of weight $\omega$ we divide it by the Groebner basis getting a decomposition
$$F=\sum H_{si}F_{si}^{(0)}+T_{\sigma j}G_{\sigma j}^{(0)}+R$$ where $R\in\Lambda_{r}$ and $H_{si}$ and $T_{\sigma j}$ are homogeneous of degree $r-2$ and $r-3$ respectively, 
and weight $\omega -s$ and $\omega-\sigma$, respectively. The remainder $R$ is the only monomial in $\Lambda_{r}$ of weight $\omega$ whose coefficient is equal to the sum
of the coefficients of $F$. Since $F\in I({\C}^{(0)})$ the sum of its coefficients is equal to zero, then $R=0$.
 \end{proof}
 
A different proof of the above lemma can be found in \cite[Thm. 1.1]{GSS} by no\-ting that the symmetric semigroup $\N=\langle0,g,g+1,\dots,2g-2\rangle$ is generated by a generalized arithmetic
sequence. So the ideal $I(\C^{(0)})$ of the monomial curve $\C^{(0)}$ is also generated by the $2\times 2$ minors of suitable two matrices. It can be seen immediately that
the ideal given by this $2\times 2$ minors is equal to the ideal generated by the binomials $F_{si}^{(0)}$ and $G_{\sigma j}^{(0)}$.

The following lemma is a generalization of result in \cite[Lemma 2.3]{CS}, where due to the assumptions the authors just deal with the first syzygies
of quadratic forms. Here we also deal with syzygies of cubic forms, which will induce nonlinear syzygies, see the equations \eqref{syzyg5}, pg. 14, and \eqref{syzyg6}, pg. 17.
  
  
 \begin{syzlem}\label{lem2}

For each of the $\frac{1}{2}(g-3)(g-4)$ quadratic forms $F_{s'i'}^{(0)}$ 
different from $F_{n_i+2g-2,1}^{(0)} (i=1,\ldots, g-3)$ there is a syzygy of the form 
\begin{equation}\label{eq2}
X_{2g-2}F_{s'i'}^{(0)}+\displaystyle\sum_{nsi}^{}\epsilon_{nsi}^{(s'i')}X_nF_{si}^{(0)}=0
\end{equation}
and for each cubic forms $G_{\sigma'j'}^{(0)}$ different from $G_{4g-4,1}^{(0)}$, there is a syzygy of the form
 \begin{equation}\label{eq3}
 X_{2g-2}G_{\sigma'j'}^{(0)}+\displaystyle\sum_{q\sigma j}^{}\rho_{q\sigma j}^{(\sigma'j')}X_qG_{\sigma j}^{(0)}=0,
 \end{equation}
 where the coefficients $\epsilon_{nsi}^{(s'i')}$, $ \rho_{q\sigma j}^{(\sigma'j')}$ are integers equal to $1, -1$ or $0$, 
 and where the sum is take over the nongaps $n, q<2g-2$, the double indexes $si$ with $s+n=2g-2+s'$ and $\sigma j$ 
 with $q+\sigma=2g-2+\sigma'$. 
 \end{syzlem}
 \begin{proof}Given a quadratic form $F=F_{s'i'}^{(0)}$ or $F=-F_{s'i'}^{(0)}$, we can write
 \begin{equation*}
 F=X_mX_n-X_qX_r,
 \end{equation*}
 where $m, n, q, r$ are nongaps satisfying $m+n=q+r$ and $q<m\leq n<r<2g-2$. If $r+1$ is a gap then, by symmetry, 
 $k:=2g-2-r+n$ is a nongap and we find the syzygy
 \begin{equation*}
X_{2g-2}(X_mX_n-X_qX_r)+X_r(X_qX_{2g-2}-X_mX_k)-X_m(X_nX_{2g-2}-X_rX_k)=0,
 \end{equation*}
 The binomials in the brackets can be written as $F_{si}^{(0)}-F_{sj}^{(0)}$, $F_{si}^{(0)}$ or $-F_{sj}^{(0)}$. 
 Analogously if $m+1$ is a gap then we take the nongap $k:=2g-2-m+r$ and we obtain a syzygy as above. Now we can 
 assume that $r+1$ and $m+1$ are nongaps, hence we have the syzygy
$$\begin{array}{l}
 X_{2g-2}(X_mX_n-X_qX_r)+X_q(X_{2g-2}X_r-X_{2g-3}X_{r+1})-\\
  X_{2g-3}(X_{m+1}X_n-X_qX_{r+1})-X_n(X_mX_{2g-2}-X_{2g-3}X_{m+1})=0.\\
 \end{array}$$
 
\noindent For a cubic form, if we put $G=G^{(0)}_{\sigma j}$ or $G=-G^{(0)}_{\sigma j}$ then we can write  
 \begin{equation*}
 G=X_mX_nX_p-X_qX_rX_t,
 \end{equation*} 
 where $m, n, p, q, r, s$ are nongaps satisfying $m+n+p=q+r+t$ and $q<m\leq n\leq r\leq p<t\leq 2g-2$. 
 
 If $p+1$ is a gap then, by symmetry, the integer $k:=2g-2-p+q$ is a nongap smaller than $2g-2$, hence we have the syzygy
 $$\begin{array}{r}
 X_{2g-2}(X_mX_nX_p-X_qX_rX_t)+X_r(X_{2g-2}X_tX_q-X_tX_pX_k)-\\
 X_p(X_{2g-2}X_mX_n-X_rX_tX_k)=0,\\
 \end{array}$$
 where the binomials in the brackets can be written as 
 $G_{\sigma j}^{(0)}-G_{\sigma i}^{(0)}$, $G_{\sigma j}^{(0)}$ or $-G_{\sigma i}^{(0)}$. 
 Analogously, if $r+1$ is a gap then $k:=2g-2-r+p$ is a nongap, and therefore we obtain the syzygy
$$\begin{array}{r}
  X_{2g-2}(X_mX_nX_p-X_qX_rX_t)+X_m(X_kX_rX_n-X_{2g-2}X_pX_n)-\\
 X_r(X_kX_mX_n-X_{2g-2}X_tX_q)=0.\\
 \end{array}$$
 Now we can assume that $p+1$ and $r+1$ are the nongaps. We just take the syzygy
 $$\begin{array}{l}
 X_{2g-2}(X_mX_nX_p-X_qX_rX_t)+X_{2g-3}(X_{r+1}X_qX_t-X_{p+1}X_nX_m) - \\ X_m(X_pX_{2g-2}X_n-X_{p+1}X_{2g-3}X_n)
- X_q(X_{2g-3}X_{r+1}X_t-X_{2g-2}X_rX_t)=0.
 \end{array}$$
 \end{proof}
 \begin{rmk} 
 \emph{The $\eta$ syzygies corresponding to the cubic forms which are multiples of the quadratics are trivial, 
 therefore we just to consider syzygies for the $\wp-1$ cubic forms, however, these $\wp-1$ syzygies are not necessarily linear.}
 \end{rmk}
 \begin{lem}\label{lem3}
 Let $I$ be the ideal generated by the $\frac{1}{2}(g-2)(g-3)$ quadratic forms $F_{si}$ and by the $\wp$ cubic forms $G_{\sigma j}$. Then,
 \begin{equation*}
 {\mathbf{k}[X_{n_0},\ldots,X_{n_{g-1}}]}_r=I_r+\Lambda_r\ \mbox{for each}\ r\geq2.
 \end{equation*}
 \end{lem}
 \begin{proof}
 Let $F$ be a homogeneous polynomial of degree $r$ and weight $w$. Let $S$ be its quasi-homogeneous component of weight $w$ and $R$ 
 the unique monomial in $\Lambda_r$ of weight $w$ whose coefficient is the sum of the coefficients of $S$. Thus, $S-R\in I(\C^{(0)})$ 
 and by Lemma \ref{lem2} we can write the expression
 \begin{equation}\label{eq1}
 S-R=\displaystyle\sum_{si}^{}S_{si}F_{si}^{(0)}+\displaystyle\sum_{\sigma j}^{}H_{\sigma j}G_{\sigma j}^{(0)}.
 \end{equation}
 Replacing each polynomial $S_{si}$ and $H_{\sigma j}$ with its homogeneous component of degree $r-2$ and $r-3$, 
 respectively, we can take $S_{si}$ and $H_{\sigma j}$ homogeneous of degree $r-2$ and $r-3$, respectively. Likewise, 
 we can assume that $S_{si}$ and $H_{\sigma j}$ are quasi-homogeneous of weight $w-s$ and $w-\sigma$, respectively. 
 Then the polynomial 
 \begin{equation*}
  F-R-\displaystyle\sum_{si}^{}S_{si}F_{si}^{(0)}-\displaystyle\sum_{\sigma j}^{}H_{\sigma j}G_{\sigma j}^{(0)}
 \end{equation*}
 is homogeneous of degree $r$ and weight smaller than $w$. Now, the proof follows by induction on $w$. 
 \end{proof}
 \begin{rmk}
 \emph{We see that if the curve $\C$ is not trigonal, then the last summand in \ref{eq1} does not appear because the 
 ideal $I(\C^{(0)})$ is generated by the $\frac{1}{2}(g-2)(g-3)$ quadratic forms $F_{si}^{(0)}$.}
 \end{rmk}


Let us now invert the considerations of the previous section. Instead of take a pointed canonical 
Gorenstein curve whose Weierstrass semigroup at the marked point is $\N:=\langle g, g+1,\ldots,2g-2\rangle$, we 
consider the semigroup $\N$ and the associated monomial curve $\C^{(0)}$. We want to deform it in order to get a 
Gorenstein curve with a marked point whose Weierstrass semigroup is also $\N$.
By Lemma \ref{lemaI0} the ideal of the monomial curve $\mathcal{C}^{(0)}$ is generated by the 
$\frac{1}{2}(g-2)(g-3)$ quadratic forms $F_{si}^{(0)}$ and by the $\wp$ cubic forms $G_{\sigma j}^{(0)}$. 
So, let us consider a \textit{pre-deformation} of the ideal of $\C^{(0)}$ which is
  \begin{equation*}\label{qddeforms}
  F_{si}=X_{a_{si}}X_{b_{si}}-X_{a_s}X_{b_s}-\displaystyle\sum_{n=0}^{s-1}c_{sin}X_{a_{n}}X_{b_{n}}
   \end{equation*} 
   and 
   \begin{equation*}\label{cdeforms}
    G_{\sigma j}=X_{a_{\sigma j}}X_{b_{\sigma j}}X_{c_{\sigma j}}-X_{a_{\sigma }}X_{b_{\sigma }}X_{c_{\sigma }}-\displaystyle\sum_{n=0}^{\sigma-1}d_{\sigma jn}X_{a_{n}}X_{b_{n}}X_{c_{n}},
     \end{equation*}
where the coefficients $c_{sin}$ and $d_{\sigma jn}$ belongs to the ground field $\k$. It is clear that we are looking for 
conditions on this coefficients such that this \textit{pre-deformation} is a 
deformation: a curve of degree $2g-2$ and genus $g$ with a marked point whose Weierstrass semigroup is $\N$. 
The main idea is to apply the Syzygy Lemma and erase the superscript zeros of the quadratic and cubic forms and then, 
by means of \eqref{directsum}, get the conditions on the coefficients.

Replacing  the left-hand side of the equation \eqref{eq2} of the Syzygy Lemma the binomials $F_{s'i'}^{(0)}$ and $F_{si}^{(0)}$ 
with the quadratic forms $F_{s'i'}$ and $F_{si}$ we obtain for each of the $\frac{1}{2}(g-3)(g-4)$ double indexes $s'i'$ a 
linear combination of cubic monomials of weight less than $s'+2g-2$, which by Lemma \ref{lem3} admits the decomposition
 \begin{equation*}
 X_{2g-2}F_{s'i'}+\displaystyle\sum_{nsi}^{}\epsilon_{nsi}^{(s'i')}X_nF_{si}=\displaystyle\sum_{nsi}^{}\eta_{nsi}^{(s'i')}X_nF_{si}+R_{s'i'},
 \end{equation*}
 where the sum on the right-hand side is taken over all the nongaps $n\leq 2g-2$, the double indexes $si$ with $n+s<s'+2g-2$, 
 the coefficients $\epsilon_{nsi}^{(s'i')}$ are constants and where $R_{s'i'}$ is a linear combination of cubic monomials 
 of pairwise different weights less than $s'+2g-2$. 
 
 Repeating the above procedure for the equation \eqref{eq3} on the Syzygy Lemma, we obtain a decomposition
 \begin{equation*}
 X_{2g-2}G_{\sigma'j'}+\displaystyle\sum_{q\sigma j}^{}\rho_{q\sigma j}^{(\sigma'j')}X_qG_{\sigma j}=
 \displaystyle\sum_{mq\sigma j}^{}\mu_{mq\sigma j}^{(\sigma'j')}X_mX_qF_{\sigma j}+\displaystyle\sum_{q\sigma j}^{}\nu_{q\sigma j}^{(\sigma'j')}X_qG_{\sigma j}+R_{\sigma'j'},
 \end{equation*}
 where the sum on the right-hand side is taken over the nongaps $m, q\leq 2g-2$, the indexes $mq\sigma$ and $q\sigma$ 
 with $m+q+\sigma<2g-2+\sigma'$ and $q+\sigma<2g-2+\sigma'$, the coefficients 
 $\mu_{mq\sigma j}^{(\sigma'j')}, \nu_{q\sigma j}^{(\sigma'j')}$ are constants and where $R_{\sigma'j'}$ is a linear 
 combination of quartic monomials of pairwise different weights less than $2g-2+\sigma'$.
 
 For each nongap $m<s'+2g+2$ (resp. $r<\sigma'+2g+2$) let $\varrho_{s'i'm}$ (resp. $\vartheta_{\sigma'j'r}$ ) 
 be the unique coefficient of $R_{s'i'}$ (resp. $R_{\sigma' j'}$) of weight $m$ (resp. $r$). We do not lost 
 information about the coefficients of $R_{s'i'}$ and $R_{\sigma' j'}$ replacing the variables $X_n$ by powers $t^n$ 
 of an indeterminate $t$. Hence it is convenient to consider the polynomials
 
 \begin{equation*}
 R_{s'i'}(t^{n_0},\ldots,t^{n_{g-1}})=\displaystyle\sum_{m=0}^{s'+2g-2}\varrho_{s'i'm}t^m
 \end{equation*} 
 and
 \begin{equation*}
  R_{\sigma'j'}(t^{n_0},\ldots,t^{n_{g-1}})=\displaystyle\sum_{r=0}^{\sigma'+2g-2}\vartheta_{\sigma'j'r}t^r.
 \end{equation*}
 
 We can assume that the coefficients $\varrho_{s'i'm}$ are quasi-homogeneous polynomial expressions of weight 
 $s'+2g-2-m$ in the constants $c_{sin}$ while the coefficients $\vartheta_{\sigma'j'r}$ are quasi-homogeneous 
 polynomial expressions of weight $\sigma'+2g-2-r$ in the constants $d_{\sigma j n}$.
 
 \begin{thm}
 Let $\N$ be a nonhyperelliptic and non-ordinary numerical symmetric semigroup of genus $g$.
 The respective $\frac{1}{2}(g-2)(g-3)$ quadratic and the $\wp$ cubic forms, $F_{si}=F_{si}^{(0)}-\sum_{n=0}^{s-1}c_{sin}X_{a_{sin}}X_{b_{sin}}$ 
 and $G_{\sigma j}=G_{\sigma j}^{{(0)}}-\sum_{n=0}^{\sigma}d_{\sigma jn}X_{a_{n}}X_{b_{n}}X_{c_{n}}$ 
 cut out a canonical integral Gorenstein curve on $\mathbb{P}^{g-1}$ if and only if the coefficients $c_{sin}, d_{\sigma jn}$
 satisfy the quasi-homogeneous equations $\varrho_{s'i'm}=0$ and $\vartheta_{\sigma'j'r}=0$. In this case, the point $P=(0:\ldots:0:1)$ 
 is a smooth point of the canonical curve with Weierstrass semigroup $\N$.
 \end{thm}
 \begin{proof}
 We first assume that the $\frac{1}{2}(g-2)(g-3)$ quadratic forms $F_{si}$ and the $\wp$ cubic forms $G_{\sigma j}$ cut out a canonical 
 curve $\C\subset\mathbb{P}^{g-1}$. Since each $R_{s'i'}$ and $R_{\sigma'j'}$ belongs to the ideal $I$, follows that 
 $R_{s'i'}(x_{n_0},\ldots,x_{n_{g-1}})=R_{\sigma'j'}(x_{n_0},\ldots,x_{n_{g-1}})=0$ for each pair of index $s'i'$ and $\sigma'j'$. 
 We can write
  \begin{equation*}
  R_{s'i'}(x_{n_0},\ldots,x_{n_{g-1}})=\displaystyle\sum_{m=0}^{s'+2g-2}\varrho_{s'i'm}z_{s'i'm}
  \end{equation*}
  and
  \begin{equation*}
   R_{\sigma'j'}(x_{n_0},\ldots,x_{n_{g-1}})=\displaystyle\sum_{r=0}^{\sigma'+2g-2}\vartheta_{\sigma'j'r}z_{\sigma'j'r},
  \end{equation*}
 where the $z_{s'i'm}, z_{\sigma'j'r}$ are monomial expressions of weights $m$ and $r$ respectively in the projective coordinates 
 functions $x_{n_0},\ldots,x_{n_{g-1}}$, and hence $z_{s'i'm}$ has pole divisor $mP$ while $z_{\sigma'j'r}$ has pole divisor $rP$. 
 Then we conclude that $\varrho_{s'i'm}=\vartheta_{\sigma'i'r}=0$.
  
 On the opposite, let us assume that the coefficients $c_{sin}, d_{\sigma jn}$ satisfy the equations $\varrho_{s'i'm}=0$ 
 and $\vartheta_{\sigma'i'r}=0$. Since the $g-3$ quadric hypersurfaces $V(F_{n_i+2g-2,1})\subset\mathbb{P}^{g-1} (i=1,\ldots,g-3)$ 
 and the cubic hypersurface $V(G_{4g-4,1})$ intersect transversally at $P$, in an open neighborhood of $P$
 their intersection has an unique irreducible component which contains $P$, and so this component is a projective integral 
 algebraic curve, say $\C$, which is smooth at $P$ and whose tangent line is the intersection of their tangent 
 hyperplanes $V(X_{n_i}), i=0,\ldots, g-3$.

  Let $y_{n_0},\ldots,y_{n_{g-1}}$ be the projective coordinate functions of $\C$ and we look for the affine open $y_{n_{g-1}}=1$. 
  Since the local coordinate ring of $C$ at $P$ is a discrete valuation ring and $n_{g-1}-n_{g-2}=l_2-l_1=1$, we have that 
  $t:=y_{n_{g-2}}$ is a local parameter of $\C$ at $P$, and $y_{n_0},\ldots, y_{n_{g-3}}$ are the power series in $t$ of order 
  greater than $1$. More precisely, comparing coefficients in the $g-3$ equations 
  $$F_{n_i+2g-2}(y_{n_0},\ldots, y_{n_{g-2}}, y_{n_{g-1}})=0,\ \ i=1,\ldots,g-3$$ and $$G_{4g-4},1(y_{n_0},\ldots, y_{n_{g-2}}, y_{n_{g-1}})=0$$ one sees that
  \begin{eqnarray*}
  y_{n_i}=t^{n_{g-1}-n_i}+ (\mbox{sum of higher orders terms}) \\= t^{l_{g-i}-1}+(\mbox{sum of higher orders terms}),
  \end{eqnarray*}
  for each integer $i=0,\ldots,g-1$. This means that the $g$ integers $l_i-1$ $(i=1,\ldots,g)$ are the contact orders of 
  the curve $\C\subset\mathbb{P}^{g-1}$ with the hyperplanes at $P$. In particular, the curve $\C$ is not contained in any hyperplane.
  
 By assumption, $\varrho_{s'i'm}=0$ and $\vartheta_{\sigma'j'r}=0$ for each pair of double indexes 
  $s'i'$ and $\sigma'j'$ , respectively. Hence, we obtain the syzygies 
   \begin{equation*}
   X_{2g-2}F_{s'i'}+\displaystyle\sum_{nsi}^{}\epsilon_{nsi}^{(s'i')}X_nF_{si}-\displaystyle\sum_{nsi}^{}\eta_{nsi}^{(s'i')}X_nF_{si}=0
   \end{equation*}
   and 
   \begin{equation*}
X_{2g-2}G_{\sigma'j'}+\displaystyle\sum_{q\sigma j}^{}\rho_{q\sigma j}^{(\sigma'j')}X_qG_{\sigma j}-
\displaystyle\sum_{mq\sigma j}^{}\mu_{mq\sigma j}^{(\sigma'j')}X_mX_qF_{\sigma j}-\displaystyle\sum_{q\sigma j}^{}\nu_{q\sigma j}^{(\sigma'j')}X_qG_{\sigma j}=0.
\end{equation*}
Replacing the variables $X_{n_0},\ldots, X_{n_{g-1}}$ by the projective coordinates functions $y_{n_0},\ldots, y_{n_{g-1}}$ 
we get two systems: a system with $\frac{1}{2}(g-3)(g-4)$ linear homogeneous equations in the $\frac{1}{2}(g-3)(g-4)$ functions 
$F_{s'i'}(y_{n_0},\ldots, y_{n_{g-1}})$ with the coefficients in the domain $\textit{k}[[t]]$ of formal power series; 
the second system is composed by $\wp-1$ linear homogeneous equations in the $\wp-1$ functions $G_{\sigma'j'}(y_{n_0},\ldots, y_{n_{g-1}})$ 
with the coefficients in the domain $\textit{k}[[t]]$ of formal power series. Since the triple indexes $nsi$ of the 
coefficients $\epsilon_{nsi}^{(s'i')}$, respectively, $\eta_{nsi}^{(s'i')}$, satisfy the inequalities $n<2g-2$ and $n+s=2g-2+s'$, 
respectively, $n\leq 2g-2$ and $n+s<2g-2+s'$, the diagonal entries of the matrix of the system have constant terms $1$, 
while the remaining entries have positive orders. Therefore, the matrix is invertible, and so the equation 
$F_{si}(y_{n_0},\ldots, y_{n_{g-1}})=0$ holds for each double index $si$. In the second system, the indexes 
$q\sigma j, mq\sigma j$ and $n\sigma j$ of the coefficients  $\rho_{q\sigma j}^{(\sigma'j')}, \mu_{mq\sigma j}^{(\sigma'j')}$ 
and $\nu_{n\sigma j}^{(\sigma'j')}$, respectively, are such that satisfy the inequalities $q<2g-2$ and $q+\sigma=2g-2+\sigma'$, 
respectively, $m, q\leq 2g-2$ and $m+q+\sigma<2g-2+\sigma'$. So the diagonal entries of the matrix of the system have constant 
terms $1$, while the remaining entries have positive orders, hence the matrix is also invertible. This means that the equation 
$G_{\sigma j}(y_{n_0},\ldots, y_{n_{g-1}})=0$ holds for each double index $\sigma j$. Therefore, we shown that $I\subset I(\C)$,
where $I$ is the ideal generated by the $\frac{1}{2}(g-2)(g-3)$ quadratic forms $F_{si}$ and by the $\wp$ cubic forms $G_{\sigma j}$. 

By virtue of Lemma \ref{lem3}, $\codim I_r\leq\dim\Lambda_r$ for each $r\geq 2$. Since $I_r(\C)\cap\Lambda_r=0$, 
we deduce $\dim\Lambda_r\leq \codim I_r(\C)$ and we obtain
\begin{equation*}
\codim I_r(\C)=\codim I_r=\dim\Lambda_r=(2g-2)r+1-g.
\end{equation*}
Thus $I(\C)=I$ and the curve $\C\subset\mathbb{P}^{g-1}$ has Hilbert polynomial $(2g-2)r+1-g$, hence $\C$ has degree $2g-2$ and 
arithmetic genus $g$.

Intersecting the curve $\C$ with the hyperplane $V(X_{2g-2})$ we obtain the divisor $D:=(2g-2)P$ of degree $2g-2$, whose 
complete linear system $|D|$ has dimension at least $g-1$, and so by Riemann--Roch theorem for complete integral 
curves the Cartier divisor $D$ is canonical, and $\C$ is a canonical Gorenstein curve.
\end{proof}

Note that the fixed $P$-hermitian basis $x_{n_0}, x_{n_1},\ldots,x_{n_{g-1}}$ of $H^0(\C,(2g-2)P)$ is 
uniquely determined up to a linear transformation $x_{n_i}\mapsto\displaystyle\sum_{j=i}^{g-1}c_{ij}x_{n_j}$, 
with $(c_{ij})\in\GL_g(\k)$ an upper triangular matrix whose diagonal entries are of the form 
$c_{ii}=c^{n_i}, i=0,\ldots, g-1$, for some non-zero constant $c$, due the normalizations $c_{sis}=1$. 
We assume that the characteristic of the field of constants $\k$ is zero (or a prime not dividing any of the 
differences $m-n$ with $n, m$ nongaps such that $m<n\leq2g-2$). 

If the symmetric semigroup is non-odd we can normalize $\frac{1}{2}g(g-1)$ coefficients $c_{sin}$ of the quadratic forms
to be zero, for each $i=1,\dots,g-1$ we just transform $$X_{n_i}\mapsto X_{n_i}+\sum_{j=1}^{i}c_{n_in_{i-j}}X_{n_{i-j}}$$ and proceed by induction
on the weight of the coefficients, as in \cite[pg. 587]{CS}. In the odd case, we can normalize $g-3$ coefficients of the cubic
form $G_{4g-4,1}$ by transforming $$X_{2g-4}\mapsto X_{2g-4}+\sum_{i=1}^{g-3}d_{2g-4,n_{g-3-i}}X_{n_{g-3-i}},$$ and by transforming 
$$X_{n_i}\mapsto X_{n_i}+\sum_{j=1}^{i}c_{n_in_{i-j}}X_{n_{i-j}}$$ with $n_i\neq n_{g-3}=2g-4$ we can normalize the remaining
$\frac{1}{2}g(g-1)-(g-3)$ coefficients of the quadratic forms $F_{n_i+2g-2,1}$.

 Due all the normalizations the only freedom left 
 to us is to transform $x_{n_i}\mapsto c^{n_i}x_{n_i}, i=0,\ldots, g-1$ for some non-zero constant $c\in\k$. 
 Therefore, we have showed:
 
 \begin{thm}\label{teo3}
 Let $\N$ be a nonhyperelliptic and non-ordinary symmetric semigroup of genus $g\geq 5$
 The isomorphism classes of the pointed complete integral Gorenstein curves with Weierstrass semigroup $\N$ 
 correspond bijectively to the orbits of the $\mathbb{G}_m(\textit{k})$-action
 \begin{equation*}
 (c,\ldots,c_{sin},\ldots)\longmapsto(\ldots,c^{s-n}c_{sin},\ldots)
 \end{equation*}
 on the  affine quasi-cone of the vectors whose coordinates are the coefficients $c_{sin}$, $d_{\sigma j n}$ 
 of the normalized quadratic and cubic forms $F_{si}$ and $G_{\sigma j}$ satisfying the quasi-homogeneous 
 equations $\varrho_{s'i'm}=\vartheta_{\sigma'i'r}=0$. 
 \end{thm}

The dimension of the moduli spaces $\M$ for any $\N$ is known for a few special cases. A great lower bound
was obtained by Pflueger in \cite{N16}, where the effective weight is an upper bound for codimension of
$\M$ in $\mathcal{M}_{g,1}$. On the other hand, an upper bound follows from a formula obtained by
Deligne \cite{Del73}. Both bounds are sharp but there are examples where the strict inequalities hold, see
\cite{NP16} and \cite{CS}. In the case of odd symmetric semigroups $\N=\langle g,g+1,\dots,2g-2\rangle$ Rim--Vitulli \cite{RV77}
showed that $\N$ is negatively graduated, hence Pflueger's lower bound and Deligne's upper bound
are equal to $2g-1=\dim \M=\dim\CM$.

\section{Odd numerical semigroups of genus at most six}

We start this section with the following observation on the rationality of $\CM$
for $\N$ symmetric and generated by less than five elements, which was also noted in \cite{CS}.
If the symmetric semigroup $\N$ is
generated by $4$ elements, 
using Pinkham's equivariant deformation theory \cite{Pi74},
Buchsbaum-Eisenbud's structure theorem
for Gorenstein ideals of codimension $3$ (see
\cite[p.\,466]{BE77}), one can deduce that the affine monomial curve $C^{(0)}$
can be negatively smoothed without any obstructions (see \cite{Bu80},
\cite{W79} \cite[Satz 7.1]{W80}), hence
$$
\CM = \proj(T^{1,-}_{\k[\N]|\k}). 
$$


\noindent Although the above observation assures that $\CM=\mathbb{P}^9$ for $\N:=\langle 5,6,7,8\rangle$, we
believe that it is relevant to illustrate our techniques in an example not so involved with large
computations.

 \subsection{Odd of genus five}
 
Let $\mathcal{C}^{(0)}$ be the canonical monomial Gorenstein curve of genus $5$ associated to the odd symmetric semigroup
of genus also $5$. Up to change of coordinates can we write:
$$\mathcal{C}^{(0)}:=\{(a^8\,:\,a^3b^5\,:\,a^2b^6\,:\,a^1b^7\,:\,b^8)\,\vert\,(a:b)\in\mathbb{P}^1\}\subseteq \mathbb{P}^4\,.$$
The symmetric Weierstrass semigroup of the smooth point $P=(0:0:0:0:1)$ is $\N:=\langle 5,6,7,8\rangle$. Following
Lemma \ref{lemaI0} the ideal of $\mathcal{C}^{(0)}$ can be generated by the following seven isobaric and homogeneous forms
$$ \begin{array}{ll}
 F_{12}^{(0)}=X_6^2-X_5X_7& F_{13}^{(0)}=X_6X_7-X_5X_8,\\
 F_{14}^{(0)}=X_7^2-X_6X_8& G_{15}^{(0)}=X_5^3-X_0X_7X_8,\\
 G_{16}^{(0)}=X_5^2X_6-X_0X_8^2& G_{18}^{(0)}=X_6^3-X_5^2X_8,\\
 G_{21}^{(0)}=X_7^3-X_5X_8^2.&\\
 \end{array}$$
 
\noindent For each nongap $n\in \N$ we take the rational function $x_n$ with pole divisor $nP$.
Writing each one of the seven rational functions $x_6^2, x_6x_7, x_7^2, x_5^3, x_5^2x_6, x_6^3$ and $x_7^3$ as linear 
combination of the basis elements of the vector spaces $H^0(\mathcal{C},2(2g-2))$ and $H^0(\mathcal{C},3(2g-2))$, 
respectively, we obtain in the variables $X_0, X_5, X_6, X_7, X_8,$ the polynomials 
    \begin{equation*}
         F_{i}=F_{i}^{(0)}-\displaystyle\sum_{j=1}^{i} c_{ij}Z_{i-j},\ (i=12, 13, 14),
         \end{equation*}
         and 
         \begin{equation*}
             G_{i}=G_{i}^{(0)}-\displaystyle\sum_{j=1}^{i} d_{ij}Z_{i-j},\ (i=15, 16, 18, 21),
             \end{equation*}
where $Z_{i-j}$ stands for the basis monomial of weight $i-j$, and the summation index $j$ varies 
only through the integers such that $i-j\in \N$.
        
By using the transformations $X_{i}\mapsto X_{i}+\sum_{j=1}^{i-1}\lambda_{j}X_{i-j}$, we can normalize the following ten coefficients
$$c_{12, 1}=c_{12, 2}=c_{12, 7}=c_{13, 1}=c_{13, 2}=c_{13, 3}=c_{13, 8}=d_{16, 1}=d_{16, 6}=d_{21, 5}=0\,.$$

              
By applying the Syzygy Lemma \ref{lem2} we obtain the following four syzygies of the canonical monomial curve $\mathcal{C}^{(0)}$
\begin{equation}\label{syzyg5}
  \begin{array}{l}
   X_8F^{(0)}_{12}-X_7F^{(0)}_{13}+X_6F^{(0)}_{14}=0,\\
   X_8G^{(0)}_{15}-X_5X_6F^{(0)}_{12}+X_5G^{(0)}_{18}-X_7G^{(0)}_{16}=0,\\
   X_8G^{(0)}_{18}-X_5G^{(0)}_{21}+X_5X_7F^{(0)}_{14}-X_6X_8F^{(0)}_{12}=0,\\
   X_8G^{(0)}_{21}-X_7X_8F^{(0)}_{14}+X_8^2F^{(0)}_{13}=0.\\
   \end{array}
\end{equation}

\noindent Replacing each left-hand side of the above syzygies the binomials $F_{s,i}^{(0)}, F_{s',i'}^{(0)},G_{\sigma, j}^{(0)}$ and 
$G_{\sigma', j'}^{(0)}$ by the quadratic and cubic forms $F_{s,i}, F_{s',i'}, G_{\sigma, j}$ and $G_{\sigma', j'}$, respectively, 
and applying the division algorithm recursively until all the monomials of these new equations belong to the basis $\Lambda_3$ or $\Lambda_4$, 
we get the following four polynomial equations 
   
\medskip

\begin{flushright}
 

\noindent $\begin{array}{r}
  X_8F_{12}-X_7F_{13}+X_6F_{14}=-F_{12}\left(c_{14, 3}X_{5}+c_{14, 8}X_{0}\right)+F_{14}c_{13, 6}X_{0}-G_{16}c_{14, 4}\\ +F_{13} \left(c_{13, 7}X_{0}-c_{14, 2}X_{5}
  -c_{14, 7}X_{0}\right),\\
\end{array}$
  
\medskip
  
\noindent $\begin{array}{r}
  X_{8}G_{15}-X_{6}G_{17}+X_{5}G_{18}-X_{7}G_{16}=(c_{12,6}X_{0}X_{5}-d_{18,1}X_{0}X_{8})F_{12}\\
  -(c_{14,3}d_{16,4}+c_{14,3}d_{15,3}d_{18,1}+d_{18,7})X_{0}G_{16}+(d_{16,5}X_{5}+c_{12,5}X_{5})X_{0}F_{13}\\
  +(d_{16,9}X_{0}-d_{18,1}X_{8}+d_{15,8}d_{18,1}X_{0}+d_{15,3}d_{18,1}X_{5}+d_{16,4}X_{5})X_{0}F_{14}\\
  +(d_{16,10}X_{0}+d_{15,9}d_{18,1}X_{0}+d_{15,1}d_{18,1}X_{8}+d_{15,4}d_{18,1}X_{5}+d_{16,2}X_{8})X_{0}F_{13}\\
  +(-c_{14,4}d_{16,4}X_{0}-c_{14,4}d_{15,3}d_{18,1}X_{0}-d_{18,1}X_{7}-d_{18,8}X_{0})G_{15},\\
\end{array}$

\medskip 
   
\noindent $\begin{array}{r}
 X_{8}G_{18}-X_{5}G_{21}-X_{6}X_8F_{12}+X_{7}X_{5}F_{14}=\\
(-c_{14, 3}^2d_{16, 4}-c_{14, 2}c_{14, 3}d_{15, 5}-c_{14, 3}c_{14, 4}d_{15, 3}+c_{14, 3}c_{14, 7})G_{16}X_{0}\\
+c_{14, 3}d_{15, 3}d_{14, 4}+c_{14, 2}c_{14,8}-c_{14, 2}c_{14, 4}d_{15, 4}+c_{14, 2}^2c_{14, 3}d_{15, 3})X_{0}G_{16}\\
(+c_{14, 4}d_{15, 9}X_{0}-d_{15, 1}c_{14, 2}^2X_{8}-d_{15, 4}d_{14, 4}X_{5}-c_{14, 8}X_{5}+c_{12, 5}X_{8}\\
+c_{14, 2}c_{14, 3}X_{8}+c_{14, 3}d_{16, 2}X_{8}-d_{15, 4}c_{14, 2}^2X_{5}-c_{14, 2}c_{14, 3}d_{15, 8}X_{0}\\
+c_{14, 4}d_{15, 1}X_{8}-d_{15, 1}d_{14, 4}X_{8}-d_{15, 9}d_{14, 4}X_{0}+c_{14, 3}d_{16, 10}X_{0}+c_{14, 3}d_{16, 5}X_{5}\\
+c_{14, 4}d_{15, 4}X_{5}-d_{15, 9}c_{14, 2}^2X_{0}-c_{14, 2}c_{14, 3}d_{15, 3}X_{5})X_{0}F_{13}\\
(+d_{14, 11}X_{0}+d_{14, 4}X_{7}+d_{14, 3}X_{8}-c_{14, 4}X_{7}+c_{14, 2}^2X_{7}-c_{14, 4}c_{14, 3}d_{16, 4}X_{0}\\
+c_{14, 1}c_{14, 2}X_{8}+d_{15, 3}c_{14, 2}^2c_{14, 4}X_{0}-c_{14, 4}^2d_{15, 3}X_{0}+c_{14, 4}d_{15, 3}d_{14, 4}X_{0}\\
-c_{14, 2}c_{14, 4}d_{15, 5}X_{0}+c_{14, 2}c_{14, 9}X_{0}+c_{14,4}c_{14, 7}X_{0}+c_{14, 2}c_{14, 4}X_{5}+c_{14, 2}c_{14, 3}X_{6})G_{15}\\
(c_{14, 2}^2X_{0}X_{8}-c_{14, 4}X_{0}X_{8}+d_{14, 4}X_{0}X_{8}-c_{14, 7}X_{0}X_{5}-c_{14, 2}X_{5}^2+c_{14, 3}d_{16, 9}X_{0}^2\\
+c_{14, 4}d_{15, 3}X_{0}X_{5}-d_{15, 3}d_{14, 4}X_{0}X_{5}+c_{14, 3}d_{16, 4}X_{0}X_{5}-d_{15, 8}d_{14, 4}X_{0}^2\\
+c_{14, 4}d_{15, 8}X_{0}^2-c_{14, 2}^2d_{15, 3}X_{0}X_{5}-c_{14, 2}^2d_{15, 8}X_{0}^2)F_{14}+(d_{14, 2}X_{8}-c_{14, 3}X_{7})G_{16}\\
(+c_{12, 6}X_{8}-c_{14, 2}c_{14, 3}d_{15, 9}X_{0}-c_{14, 2}c_{14, 3}d_{15, 4}X_{5}-c_{14, 2}c_{14, 3}d_{15, 1}X_{8})X_{0}F_{12},\\
 \end{array}$

 \medskip
   
 \noindent $\begin{array}{r}
 G_{21,2}X_{8}-G_{21,1}X_{8}-G_{22}X_{7}=X_8(c_{14, 3}X_{5}+c_{14, 8}X_{0})F_{13}\\
 +X_{8}\left[(c_{14, 2}X_{5}+c_{14, 7}X_{0})F_{14}-c_{14, 2}c_{14, 4}G_{15}-c_{14, 2}c_{14, 3}G_{16}\right].\\
\end{array}$

 \end{flushright}
 

  \medskip
 
\noindent We now determine the weighted vector space $T^{1,-}_{\k[\N]|\k}$, which is (up to an isomorphism) the 
locus of the linearizations of the above 4 equations, all we have to do is substituting by zero the right 
hand side of each equation. These four equations give rise to other $20$ linear equations obtained by 
replacing $X_{n_i}\mapsto t^{n_i}$. We can solve this linear system as follows:

 \medskip
\noindent $\begin{array}{l}
d_{16,10}=d_{15,10}, d_{16,9}=d_{15,9}, d_{16,8}=d_{15,8}, c_{14,7}=c_{13,7}, d_{18,7}=c_{13,7}, d_{15,7}=-c_{13,7}, \\
d_{21,7}=2c_{13,7},c_{14,6}=-c_{12,6}, d_{21,6}=-c_{12,6}, d_{18,6}=c_{12,6}, d_{16,5}=d_{15,5},\\
c_{14,4}=-c_{12,4}, d_{16,4}=d_{15,4}, d_{21,4}=-c_{12,4}, d_{18,4}=c_{12,4}, d_{16,3}=d_{15,3}, d_{16,2}=d_{15,2}.\\
\end{array}$

  \medskip
 
\noindent We can verify that the weighted vector space $T^{1,-}_{\mathbf{k[\N]}|\mathbf{k}}$ depends only on the ten coe\-fficients 
$d_{15,2}, d_{15,3}, , c_{12,4}, d_{15,4}, d_{15,5}, c_{12,6}, c_{13,7}, d_{15,8}, d_{15,9}, d_{15,10}\,$,
 which implies   $$\dim T^{1,-}_{\mathbf{k[\N]}|\mathbf{k}}=10.$$
 More precisely, counting the coefficients of weight $s$, we obtain the dimension of the graded component
 of $T^{1,-}_{\mathbf{k[\N]}|\mathbf{k}}$ of negative weight $-s$:
 
$$\begin{array}{l}
 \dim T^{1,-}_{s}=1,\ (s=-10, -9, -8, -7, -6, -5, -3,-2)\ \mbox{and}\ \dim T^{1,-}_{-4}=2.\\
  \end{array}$$
For the remainder integers the dimension of $T^{1,-}_{s}$ is zero. In particular, 
the compactified moduli space $\overline{\mathcal{M}_{5,1}^{\N}}$ can be 
realized as closed subspace of the $9$-dimensional weighted projective space 
$\mathbb{P}\left(T^{1,-}_{\mathbf{k[\N]}|\mathbf{k}}\right).$ 

Finally, we solve the four polynomial equations of the previous page to obtain the equations of the moduli 
variety $\overline{\mathcal{M}_{5,1}^{\N}}$. By replacing $X_{n_i}\mapsto t^{n_i}$
the compactified moduli space $\overline{\mathcal{M}_{5,1}^{\N}}$ is cut out by $70$ 
equations which depend on $64$ variables, we can solve them in the following way:

\begin{itemize}
 \item 18 coefficients which are identically zero, namely:
\end{itemize}
$$\begin{array}{l}
c_{12,5}=c_{13,5}=c_{13,6}=c_{14,1}=c_{14,2}=c_{14,3}=d_{15,1}=d_{16,11}=d_{18,1}=0 \\
d_{18,2}=d_{18,3}=d_{18,5}=d_{18,8}=d_{18,11}=d_{21,1}=d_{21,2}=d_{21,3}=d_{21,10}=0.\\ 
\end{array}$$

\medskip

\begin{itemize}
 \item  11 linear equations:
\end{itemize}
$$\begin{array}{llll}
 c_{14,4}=-c_{12,4}, & d_{15,7}=-c_{13,7}, & d_{16,2}=d_{15,2}, & d_{16,4}=d_{15,4}, \\
 d_{16,5}=d_{15,5}, & d_{16,9}=d_{15,9}, & d_{18,4}=c_{12,4},  & d_{18,6}=c_{12,6}, \\
 d_{18,7}=c_{13,7}, & d_{21,4}=-c_{12,4}, & d_{16,3}=d_{15,3}. & \\
\end{array}$$

\medskip

\begin{itemize}
 \item 17 quadratic polynomials and isobarics:
\end{itemize}
$$\begin{array}{lll}
c_{12,12}=-c_{12,4}d_{15,8}, & & c_{13,13}=c_{12,4}d_{15,9}, \\
c_{14,6}=-c_{12,4}d_{15,2}-c_{12,6}, & & c_{14,7}=-c_{12,4}d_{15,3}+c_{13,7}, \\
c_{14,8}=-c_{12,4}d_{15,4} & &c_{14,9}=-c_{12,4}d_{15,5},\\
c_{14,14}=-c_{12,4}d_{15,10}, & & d_{15,15}=c_{12,6}d_{15,9}+c_{13,7}d_{15,8},  \\
d_{16,8}=-c_{12,4}d_{15,4}+d_{15,8},& & d_{16,10}=-c_{12,6}d_{15,4}-c_{13,7}d_{15,3}+d_{15,10}, \\ 
d_{18,12}=-c_{12,4}d_{15,8}-c_{12,6}^{2}, & & d_{21,6}=-c_{12,4}d_{15,2}-c_{12,6}, \\
d_{18,13}=c_{12,4}d_{15,9}, & & d_{21,7}=-c_{12,4}d_{15,3}+2\,c_{13,7}, \\
d_{21,8}=-c_{12,4}d_{15,4}, & & d_{21,9}=-c_{12,4}d_{15,5},\\
 d_{18,10}=-c_{12,4}c_{12,6}. & & \\
\end{array}$$

\medskip

\begin{itemize}
 \item and the following 8:
\end{itemize}
$$ \begin{array}{l}
d_{16,16}=-c_{12,4}d_{15,3}d_{15,9}+c_{12,4}d_{15,4}d_{15,8}, \ \ \ \ d_{18,18}=c_{12,4}c_{12,6}d_{15,8},\\
d_{21,13}=-c_{12,4}^{2}d_{15,2}d_{15,3}-c_{12,4}c_{12,6}d_{15,3}+c_{12,4}c_{13,7}d_{15,2}+c_{12,4}d_{15,9}+c_{12,6}c_{13,7},\\  
d_{21,14}=-c_{12,4}^{2}d_{15,3}^{2}+2\,c_{12,4}c_{13,7}d_{15,3}-c_{12,4}d_{15,10}-c_{13,7}^{2},\\
d_{21,15}=-c_{12,4}^{2}d_{15,3}d_{15,4}+2\,c_{12,4}c_{13,7}d_{15,4}, \ \ \ \ d_{21,11}=-{c_{12,4}}^{2}d_{15,3}+c_{12,4}c_{13,7},\\
d_{21,16}=-c_{12,4}^{2}d_{15,3}d_{15,5}+c_{12,4}c_{13,7}d_{15,5},\\
d_{21,21}=-c_{12,4}^{2}d_{15,3}d_{15,10}+c_{12,4}^{2}d_{15,4}d_{15,9}+c_{12,4}c_{13,7}d_{15,10}.\\
\end{array}$$

\medskip

\noindent We note that there are 16 missing equations from the 70 annunciated, but each one of these 16 is redundant.
We also note that no one condition on the $10$ coefficients of the ambient space $T^{1,-}_{\mathbf{k[\N]}|\mathbf{k}}$ appears, which means
$$\overline{\mathcal{M}_{5,1}^{\N}}=\mathbb{P}(T^{1,-}_{\mathbf{k[\N]}|\mathbf{k}})\cong\proj^9_{\alpha},\ \mbox{ with } \alpha=(2,3,4,4,5,6,7,8,9,10).$$
  
 \subsection{Odd of genus six}
 
Let $\mathcal{C}^{(0)}$ be the canonical monomial Gorenstein curve of genus $6$ associated to the odd symmetric semigroup
$\N:=<6,7,8,9,10>$. Take $P=(0:0:0:0:0:1)$ a smooth point in $\C^{(0)}$ whose Weierstrass semigroup is $\N$. 
Applying Lemma \eqref{lemaI0}, the generators of the ideal of $\C^{(0)}$ are the following $6$ quadratic and $8$ cubic forms:
$$ \begin{array}{lll}
 F_{14}^{(0)}=X_7^2-X_6X_8& F_{15}^{(0)}=X_7X_8-X_6X_9& F_{16}^{(0)}=X_8^2-X_6X_{10},\\
  F_{16,1}^{(0)}=X_7X_9-X_6X_{10}& F_{17}^{(0)}=X_8X_9-X_7X_{10}& F_{18}^{(0)}=X_9^2-X_8X_{10},\\
G_{18}^{(0)}=X_6^3-X_0X_8X_{10}& G_{19}^{(0)}=X_6^2X_7-X_0X_9X_{10}& G_{20}^{(0)}=X_6^2X_8-X_0X_{10}^2,\\
 G_{20,1}^{(0)}=X_6X_7^2-X_0X_{10}^2& G_{21}^{(0)}=X_7^3-X_6^2X_9& G_{22}^{(0)}=X_7^2X_8-X_6^2X_{10},\\
  G_{26}^{(0)}=X_8X_9^2-X_6X_{10}^2& G_{27}^{(0)}=X_9^3-X_7X_{10}^2.\\  
 \end{array}$$

\noindent We consider a \textit{pre-deformation} of the ideal of $\C^{(0)}$ as follows:
  \begin{equation*}
     F_{i}=F_{i}^{(0)}-\displaystyle\sum_{j=1}^{i} c_{ij}Z_{i-j},\ (i=14,\ldots,18\mbox{ and } i=16,1)
     \end{equation*}
     and 
     \begin{equation*}
         G_{i}=G_{i}^{(0)}-\displaystyle\sum_{j=1}^{i} d_{ij}Z_{i-j},\ (i=18,\ldots, 22, 26, 27\mbox{ and }i=20,1).
         \end{equation*}
where $Z_{i-j}$ is a polynomial of weight $i-j$, whenever $i-j$ is a nongap of $\N$. 
By suitable transformations of the variables $X_0, X_6, X_7, X_8, X_9, X_{10}$, we are able to normalize the following $15$ coefficients:
  $$\begin{array}{l}
  c_{14,1}=c_{15,1}=c_{16,1,1}=d_{18,1}=d_{18,2}=c_{15,2}=c_{16,1,2}=c_{15,3}=0,\\
  c_{16,1,3}=c_{16,1,4}=c_{15,6}=c_{14,7}=c_{14,8}=c_{15,9}=c_{16,1,10}=0.\\
  \end{array}$$
We also consider the ten syzygies of the monomial curve $\C^{(0)}$, which are induced by the Syzygy Lemma \eqref{lem2}.
\begin{equation}\label{syzyg6}  
\begin{array}{l}
X_{10}F_{14}^{(0)}-X_8F_{16,1}^{(0)}+X_7F_{17}^{(0)}=0,\\
X_{10}F_{15}^{(0)}-X_9F_{16,1}^{(0)}+X_7F_{18}^{(0)}=0,\\
X_{10}F_{16}^{(0)}-X_{10}F_{16,1}^{(0)}-X_9F_{17}^{(0)}+X_8F_{18}^{(0)}=0,\\
X_{10}G_{18}^{(0)}-X_8G_{20}^{(0)}+X_6^2F_{16}^{(0)}=0,\\
 X_{10}G_{19}^{(0)}-X_9G_{20,1}^{(0)}+X_6X_7F_{16,1}^{(0)}=0,\\
 X_{10}G_{20}^{(0)}-X_{10}G_{20,1}^{(0)}+X_6X_{10}F_{14}^{(0)}=0,\\
X_{10}G_{21}^{(0)}-X_7X_{10}F_{14}^{(0)}-X_6X_{10}F_{15}^{(0)}=0,\\
X_{10}G_{22}^{(0)}-X_6X_{10}F_{16}^{(0)}-X_8X_{10}F_{14}^{(0)}=0,\\
X_{10}G_{26}^{(0)}-X_{10}^2F_{16,1}^{(0)}-X_9X_{10}F_{17}^{(0)}=0,\\
X_{10}G_{27}^{(0)}-X_{10}^2F_{17}^{(0)}-X_9X_{10}F_{18}^{(0)}=0.\\
\end{array}
\end{equation}

\medskip

\noindent The $10$ above syzygies of the monomial curve give rise to $10$ polynomial equations 
between the $14$ polynomials $F_{i}'s$ and $G_{j}'s$.
\begin{equation}\label{presyzy6}
\begin{array}{l}
X_{10}F_{14}-X_8F_{16,1}+X_7F_{17},\\
X_{10}F_{15}-X_9F_{16,1}+X_7F_{18},\\
X_{10}F_{16}-X_{10}F_{16,1}-X_9F_{17}+X_8F_{18},\\
X_{10}G_{18}-X_8G_{20}+X_6^2F_{16},\\
 X_{10}G_{19}-X_9G_{20,1}+X_6X_7F_{16,1},\\
 X_{10}G_{20}-X_{10}G_{20,1}+X_6X_{10}F_{14},\\
X_{10}G_{21}-X_7X_{10}F_{14}-X_6X_{10}F_{15},\\
X_{10}G_{22}-X_6X_{10}F_{16}-X_8X_{10}F_{14},\\
X_{10}G_{26}-X_{10}^2F_{16,1}-X_9X_{10}F_{17},\\
X_{10}G_{27}-X_{10}^2F_{17}-X_9X_{10}F_{18}.\\
\end{array}
\end{equation}
Again, we compute the linearization of the above ten polynomials, which is isomorphic to the weighted 
vector space $T^{1,-}_{\mathbf{k[\N]}|\mathbf{k}}$. To do this, we make the substitutions $X_i\mapsto t^i$ and
solve a homogeneous linear system with $60$ equations. We can solve it in way that the solution depends only on the $15$ coefficients:
$$\begin{array}{cccccccc}
d_{18,12}, & d_{18,11}, & c_{15,8}, & c_{16,1,9}, & c_{16,1,8},& c_{15,7}, & c_{14,6}, & d_{18,6},\\
d_{18,10}, & c_{14,5}, & d_{18,5}, & c_{14,4},& d_{18,4}, & d_{18,3}, & c_{14,2}. & \\
\end{array}$$

\noindent Therefore the compactified moduli space $\overline{\mathcal{M}_{6,1}^{\N}}$
can be realized as a closed subset of the $14$-dimensional weighted projective 
space $\mathbb{P}(T^{1,-}_{\mathbf{k[\N]}|\mathbf{k}})\cong\proj^{14}_{\alpha}$, where
$\alpha=(2,3,4,4,5,5,6,6,7,8,8,9,10,11,12)$. Since the odd symmetric semigroup $\N$ is ne\-ga\-ti\-vely graded, cf. \cite{RV77},
the moduli space $\mathcal{M}_{6,1}^{\N}$ has codimension three in $\mathcal{M}_{6,1}$, cf. \cite{N16,Del73}. Hence
$\overline{\mathcal{M}_{6,1}^{\N}}$ has dimension $11$.

\noindent Now we have to take each polynomial in \eqref{presyzy6} and make successive divisions in order that 
all its monomials belongs to the basis $\Lambda_3$ or $\Lambda_4$, it is possible by virtue of Lemma \ref{lem3}.
This procedure is completely computational and we can make it by using a suitable software on computer algebra, like
Singular or Maple. Here we do not display the resulting polynomials, just because they have a huge number of monomials.
Then, we make the substitutions $X_i\mapsto t^i$, with $i=6,7,8,9,10$, on the $10$ polynomials whose monomials
are in $\Lambda_3$ and $\Lambda_4$ and solve $188$ polynomial equations. This system can be solved by 
increasing weights whose solution depends only on the $15$ coefficients of the linearization that here we rename them:
\begin{equation*}
 \begin{array}{l}
  d_{18,i}:=b_{i}\ \ (i=3, 4, 5, 6, 10, 11, 12),\\ 
  c_{14,j}:=a_j\ \ (j=2, 4, 5, 6), \\
  c_{16,1,8}:=b_8, \ \ c_{15,7}:=a_7,\ \ c_{15,8}:=a_8,\ \ c_{16,1,9}:=a_9. 
 \end{array}
\end{equation*}
By Theorem \ref{teo3} we can conclude that the moduli space
$\overline{\mathcal{M}_{6,1}^{\N}}$ is given by the zero locus of following $5$ isobaric polynomials.


\medskip

\noindent$\begin{array}{l}
\vartheta_{15}:={{4}}a_{{5}}a_{{6}}-a_{{2}}a_{{5}}b_{{8}}+a_{{4}}a_{{5}}b_{{6}}-a_{{4}}b_{{3}}b_{{8}}+{a_{{5}}}^{3}+{a_{{5}}}^{2}b_{{5}}+a_{{4}}b_{{11}}+a_{{5}}b_{{10}}+2\,a_{{7}}b_{{8}}.\\
\end{array}$

\bigskip

\noindent $\begin{array}{l}
\vartheta_{13}:=2\,a_{{2}}a_{{5}}a_{{6}}+{a_{{4}}}^{2}a_{{5}}+{a_{{4}}}^{2}b_{{5}}+a_{{4}}a_{{5}}b_{{4}}+a_{{4}}a_{{6}}b_{{3}}+{a_{{5}}}^{2}b_{{3}}-a_{{4}}a_{{9}}+a_{{5}}a_{{8}}-a_{{5}}b_{{8}}\\
-2\,a_{{6}}a_{{7}}.\\
\end{array}$

\bigskip

\noindent $\begin{array}{l}
\vartheta_{17}:=a_{{5}}b_{{12}}-a_{{2}}{a_{{5}}}^{3}-a_{{2}}{a_{{5}}}^{2}b_{{5}}-a_{{4}}a_{{5}}b_{{8}}-a_{{4}}b_{{5}}b_{{8}}+2\,{a_{{5}}}^{2}a_{{7}}+a_{{5}}a_{{7}}b_{{5}}-a_{{5}}a_{{8}}b_{{4}}\\-a_{{5}}a_{{9}}b_{{3}}
-a_{{6}}b_{{11}}+a_{{9}}b_{{8}}.\\
\end{array}$

\bigskip

\noindent $\begin{array}{l}
\vartheta_{16}:=a_{{2}}a_{{4}}{a_{{5}}}^{2}+a_{{2}}a_{{4}}a_{{5}}b_{{5}}-a_{{2}}a_{{6}}b_{{8}}-2\,a_{{4}}a_{{5}}a_{{7}}-a_{{4}}{a_{{6}}}^{2}-a_{{4}}a_{{6}}b_{{6}}-a_{{4}}a_{{7}}b_{{5}}+{b_{{8}}}^{2}\\
+a_{{4}}a_{{8}}b_{{4}}+a_{{4}}a_{{9}}b_{{3}}-a_{{4}}b_{{4}}b_{{8}}-{a_{{5}}}^{2}a_{{6}}-a_{{5}}a_{{6}}b_{{5}}-a_{{5}}b_{{3}}b_{{8}}-a_{{4}}b_{{12}}-a_{{6}}b_{{10}}-a_{{8}}b_{{8}}.\\
\end{array}$

\medskip

\noindent $\begin{array}{l}
\vartheta_{19}:={a_{{2}}}^{2}{a_{{5}}}^{3}+{a_{{2}}}^{2}{a_{{5}}}^{2}b_{{5}}+a_{{2}}a_{{4}}{a_{{5}}}^{2}b_{{3}}+a_{{2}}a_{{4}}a_{{5}}b_{{3}}b_{{5}}-4\,a_{{2}}{a_{{5}}}^{2}a_{{7}}-3\,a_{{2}}a_{{5}}a_{{7}}b_{{5}}\\
+b_{{8}}b_{{11}}+a_{{2}}a_{{5}}a_{{8}}b_{{4}}+a_{{2}}a_{{5}}a_{{9}}b_{{3}}+{a_{{4}}}^{2}a_{{5}}a_{{6}}+{a_{{4}}}^{2}a_{{5}}b_{{6}}+{a_{{4}}}^{2}a_{{6}}b_{{5}}+{a_{{4}}}^{2}b_{{5}}b_{{6}}+a_{{4}}{a_{{5}}}^{3}\\
-a_{{9}}b_{{10}}+2\,a_{{4}}{a_{{5}}}^{2}b_{{5}}-2\,a_{{4}}a_{{5}}a_{{7}}b_{{3}}+a_{{4}}a_{{5}}{b_{{5}}}^{2}-a_{{4}}a_{{7}}b_{{3}}b_{{5}}+a_{{4}}a_{{8}}b_{{3}}b_{{4}}+a_{{4}}a_{{9}}{b_{{3}}}^{2}\\
-a_{{2}}a_{{5}}b_{{12}}-a_{{2}}a_{{6}}b_{{11}}+a_{{4}}a_{{5}}b_{{10}}-a_{{4}}a_{{6}}a_{{9}}-a_{{4}}a_{{9}}b_{{6}}-a_{{4}}b_{{3}}b_{{12}}-a_{{4}}b_{{4}}b_{{11}}+a_{{4}}b_{{5}}b_{{10}}\\
-{a_{{5}}}^{2}a_{{9}}+4\,a_{{5}}{a_{{7}}}^{2}-a_{{5}}a_{{9}}b_{{5}}-a_{{5}}b_{{3}}b_{{11}}+2\,{a_{{7}}}^{2}b_{{5}}-2\,a_{{7}}a_{{8}}b_{{4}}-2\,a_{{7}}a_{{9}}b_{{3}}+2\,a_{{7}}b_{{12}}\\
-a_{{8}}b_{{11}}.\\
\end{array}$

\bigskip

By intersecting $\overline{\mathcal{M}_{6,1}^{\N}}$ with the open affine chart $\{a_5=1\}$ of $\mathbb{P}^{15}$, we see that $\overline{\mathcal{M}_{6,1}^{\N}}$
admits the following local parametrization

\medskip

\noindent $\begin{array}{l}
b_{10}=a_{{4}}b_{{3}}b_{{8}}+a_{{2}}b_{{8}}-a_{{4}}a_{{6}}-a_{{4}}b_{{6}}-a_{
{4}}b_{{11}}-2\,a_{{7}}b_{{8}}-b_{{5}}-1\\
b_{12}=a_{{4}}b_{{5}}b_{{8}}+a_{{2}}b_{{5}}+a_{{4}}b_{{8}}+a_{{6}}b_{{11}}-a_
{{7}}b_{{5}}+a_{{8}}b_{{4}}+a_{{9}}b_{{3}}-a_{{9}}b_{{8}}+a_{{2}}-2\,a
_{{7}}\\
b_8={a_{{4}}}^{2}b_{{5}}+a_{{4}}a_{{6}}b_{{3}}+2\,a_{{2}}a_{{6}}+{a_{{4}}}
^{2}-a_{{4}}a_{{9}}+a_{{4}}b_{{4}}-2\,a_{{6}}a_{{7}}+a_{{8}}+b_{{3}}.
\end{array}$

\medskip

\noindent Since $\mathcal{M}_{6,1}^{\N}$ is irreducible \cite[Thm 1.1]{Bu13}, the moduli variety $\mathcal{M}_{6,1}^{\N}$ is rational of dimension $11$.
We also note that Bullock \cite[Thm. 1]{Bu14} proved that the moduli spaces $\mathcal{M}_{g,1}^{\N}$ are stably rationals when $2\leq g\leq 6$,
with the possible exceptions $<6,7,8,9,10>$ and $<5,7,8,9,11>$, the last one is not subcanonical.

For a given monomial curve $\C$ associated to a semigroup $\N$, its obstruction space lies in the
second cohomological module of cotangent complex $T^2:=T^2(\k[\N]|\k)$. As noted in the beginning of the last section of this work,
if $\N$ is symmetric and generated by less than five elements, the monomial curve $\C$ can be smoothed without
any obstructions, which implies that $\CM$ is the weighted projective space $\proj(T^{1,-}(\k[\N]|\k))$.  
The obstructions spaces of the two examples of this section are nonzero. To see this,
we use the description of $T^2$ given by Buschweitz in \cite[Thm 2.3.1]{Bu80}, and we can conclude that for genus
five, $\N=<5,6,7,8>$, the homogeneous graded part of degree $-9$ of $T^2$ has dimension $1$, for genus $6$, $\N=<6,7,8,9,10>$,
the homogeneous graded part of degree $-13$ has dimension $1$, and in both cases $T^1$ and $T^2$ are negatively graded.   

\end{document}